\documentclass[a4paper,10pt]{article}
\usepackage[T1]{fontenc}
\usepackage[english]{babel}
\usepackage{amssymb,amsmath,latexsym,epsfig}
\usepackage[all]{xy}
\xyoption{all}
\usepackage[latin1]{inputenc}
\usepackage[T1]{fontenc}
\usepackage[refpage]{nomencl}

\title{Hamming weights and Betti numbers of Stanley-Reisner rings associated to matroids\thanks{The original publication is available at http://link.springer.com/article/10.1007/s00200-012-0183-7}}
\author{Trygve Johnsen\thanks{ Dept. of Mathematics, University of Troms{\o}, N-9037 Troms{\o}, Norway, \texttt{Trygve.Johnsen@uit.no}} \and Hugues Verdure\thanks{ Dept. of Mathematics, University of Troms{\o}, N-9037 Troms{\o}, Norway, \texttt{Hugues.Verdure@uit.no}} }
\newtheorem{definition}{Definition}[section]
\newtheorem{remark}{Remark}[section]
\newtheorem{lemma}{Lemma}[section]
\newtheorem{corollary}{Corollary}[section]
\newtheorem{proposition}{Proposition}[section]
\newtheorem{theorem}{Theorem}[section]
\newtheorem{example}{Example}[section]

\newenvironment{proof}[1][Proof]{\begin{trivlist}
\item[\hskip \labelsep {\bfseries #1}]}{\end{trivlist}}

\newcommand{\qed}{\nobreak \ifvmode \relax \else
      \ifdim\lastskip<1.5em \hskip-\lastskip
      \hskip1.5em plus0em minus0.5em \fi \nobreak
      \vrule height0.75em width0.5em depth0.25em\fi}

\begin{document}

\maketitle

\begin{abstract}\noindent 
To each linear code $C$ over a finite field we associate the matroid $M(C)$ of its parity check matrix. For any matroid $M$ one can define its generalized Hamming weights, and if a matroid is associated to such a parity check matrix, and thus of type $M(C)$, these weights are the same as those of the code $C$. In our main result we show how the weights $d_1,\cdots,d_k$ of a matroid $M$ are determined by the $\mathbb{N}$-graded Betti numbers of the Stanley-Reisner ring of the simplicial complex whose faces are the independent sets of $M$, and derive some consequences. We also give examples which give negative results concerning other types of (global) Betti numbers, and using other examples we show that the generalized Hamming weights do not in general determine the $\mathbb{N}$-graded Betti numbers of the Stanley-Reisner ring. The negative examples all come from matroids of type $M(C)$.
\end{abstract}

\newcommand{\Proj}{\mathbb{P}}
\newcommand{\N}{\mathbb{N}}
\renewcommand{\k}{\Bbbk}
\newcommand{\Fq}{\mathbb{F}_q}
\newcommand{\K}{\mathbb{K}}
\newcommand{\I}{\mathcal{I}}
\newcommand{\B}{\mathcal{B}}
\newcommand{\F}[1]{\mathbb{F}_{#1}}

%%%%%%%%%%%%%%%%%%%%%%%%%%%%%%%%%%%%%%%%%%%%%%%%%%%%%%%%%%%%%%%%%%%%%%%%%%%%%%%%%%%%%%%%%%%%%%%%%%%%%%%%%%%%
%
%  INTRODUCTION
%
%%%%%%%%%%%%%%%%%%%%%%%%%%%%%%%%%%%%%%%%%%%%%%%%%%%%%%%%%%%%%%%%%%%%%%%%%%%%%%%%%%%%%%%%%%%%%%%%%%%%%%%%%%%

\section{Introduction}

Let $\Fq$ be a finite field. A linear code $C$ is a linear subspace of $\Fq^n$  for some $n \in{\N}$. We usually denote the dimension of the code by $k$ (and it can be defined as $k = log_q |C|$ also for non-linear block codes). Such a code is called a $[n,k]$-code over $\Fq$. For $h = 1, 2,\ldots, k$ let $D_h$ be the set of all linear subspaces of the linear code $C$ of dimension $h$, and let \[d_h = \min \{|\textrm{Supp}(E)| :\ E \in D_h\}.\]
As usual  $d_1$ can be identified with the minimum distance \[d = \min_{{\bf x}, {\bf y} \in C, {\bf x} \neq {\bf y}}d({\bf x}, {\bf y}),\] of the code $C$, where $d({\bf x}, {\bf y})$  is the usual Hamming distance. One aim in coding theory is to maximise $d$ given $q, n, k$. In the processes of trellis decoding, (or in certain methods in cryptology, using generator/parity check matrices of $C$ instead as a starting point), it is interesting to study and maximise $d_h$ also for higher values of $h$. Thus a full determination of the code parameters for a linear code over $\Fq$ can be said to involve finding $n,k, d_1,\ldots,d_k$.

These parameters are completely determined by the underlying matroid structure of the code. All generator matrices $G$ for the code determine the same finite matroid $M_G$ of rank $k$ and cardinality $n$ for an $[n,k]$-code $C$, while all parity check matrices $H$ determine the same matroid $M_H$ of rank $r=n-k$ and cardinality  $n$. We shall say that the matroid associated to the code is $M(C)=M_H$ where $H$ is a parity check matrix. This is independent of the choice of $H$.

Given a parity check matrix $H$ it is a well established fact that: $d_h$ is the minimum number $s$, such that there are $s$ columns of $H$ that form a submatrix of rank $s-h$. 

It is a well known fact that the matroids $M_H$ and $M_G$ are matroid duals: hence $M(C)$ and $M(C^\perp)$ are matroid dual (and determine each other), where $C^\perp$ is the orthogonal complement of $C$. Furthermore the weight hierarchy $d_1,\ldots,d_k$ of $C$ determines the weight hierarchy $d_1^\perp,\ldots,d_r^\perp$ of $C^\perp$ by Wei duality (See \cite{W}), and vice versa.

In this article we show to what extent the code parameters (in particular the generalized Hamming weights) of a linear code are determined by the various sets of Betti numbers one can associate to simplicial complexes derived from the underlying  matroid $M(C)$, through matroid and Alexander duality.

The aim of the paper is to show that the information embedded in a minimal $\N$-graded resolution of the Stanley-Reisner ring associated to a matroid/code contains its weight hierarchy in a non-trivial manner. When this resolution is indeed computable, one gets the higher weights $d_i$ for free also. In this way we want to build a bridge between the extensive activity within combinatorial algebra, around simplicial complexes and  monomial ideals, on one hand, and coding theory on the other, where the importance of generalized Hamming weights of linear codes is well established.  The concept which connects these seemingly different activities is that of matroids, and the results that we give, in particular Theorem \ref{main}, are therefore formulated for matroids, which (by now) is the most general setting for such results. We don't, however, claim having found a cheaper way to compute the higher weights for linear codes than by using the standard techniques for doing so. See the subsection below on computational complexity, where we analyze and compare with previous results on complexity.

\subsection{Structure of the paper}

In Section 2 we recall the standard definitions of Stanley-Reisner rings of simplicial complexes, their resolutions and Betti numbers of various kinds. We also recall the standard definition of matroids. In addition we describe their nullity functions and define the higher weights of matroids, which is a generalization of higher weights of codes.

In Section 3 we define sets of so-called non-redundant circuits of matroids $(E,\I)$ and use properties of matroids to identify the nullity of a subset of $E$ with the maximal number of non-redundant circuits contained in it.

In Section 4 we give our main results, both in positive and negative direction. The first part of this section explains how the higher weights of a matroid are determined by the $\N$-graded Betti numbers of its associated Stanley-Reisner ring. In a second part, we give examples showing that the converse doesn't hold, and that it doesn't hold for global Betti numbers either. Finally, in the third part of this section, we investigate Alexander duality, which gives us simpler resolutions, but show that the ${\N}$-graded and global Betti-numbers of these resolutions are not sufficient to give us the weight hierarchy.

In Section 5 we derive some consequences of our results, concerning MDS-properties of linear codes, and some codes from algebraic curves.

\subsection{Main results}

There are $12$ sets of Betti numbers associated to a matroid $M$, namely, the finely (or $\N^E-$) graded, coarsely (or $\N-$) graded and global Betti numbers of the $4$ simplicial complexes $M$, $\overline{M}$, $M^\star$ and $(\overline{M})^\star$ (the matroid itself, its dual, its Alexander dual, and the Alexander dual of its dual). The finely graded Betti numbers of any of these simplicial complexes trivially entirely determine $M$. We show that the coarsely graded Betti numbers of $M$ and $\overline{M}$ determine the weight hierarchy in a non-trivial way(Theorem~\ref{main}). We also give examples that show that the converse doesn't hold, and that none of the (ordered) sets of global Betti numbers defined determine the higher weights.

\subsection{Remarks on computational complexity}
Vardy's paper \cite{V} showed that computing the minimum distance of a binary linear code is an NP-hard problem. One might ask whether our main result indicates that computing $\N$-graded Betti numbers of Stanley-Reisner rings of simplicial complexes also is NP-hard, since doing that is an important part in our procedure to obtain the higher weights, and thereby the minimum distance of a code/matroid.
 
The situation is, however, complex. From the input (code given by its parity check matrix) to the output (minimum distance  $d_1$), there is in our set-up in principle $4$ steps, as illustrated in our diagram:
 
Step $1$ consists of determining (the dependent sets of) the matroid derived from the parity check matrix.
 
Step $2$ consists of determining the Stanley-Reisner ring/ideal of the simplicial complex which is the output of Step $1$.
 
Step $3$ is to compute the $\N$-graded Betti numbers of the Stanley-Reisner ring (ideal) which is the ouput of Step 2. Algorithms for computing Betti numbers are known, for example~\cite{Bigatti1,Roune1,Saenz1}.
 
Step $4$ is to find the minimum distance from the $\N$-graded Betti numbers.
 
The following diagram summarizes that
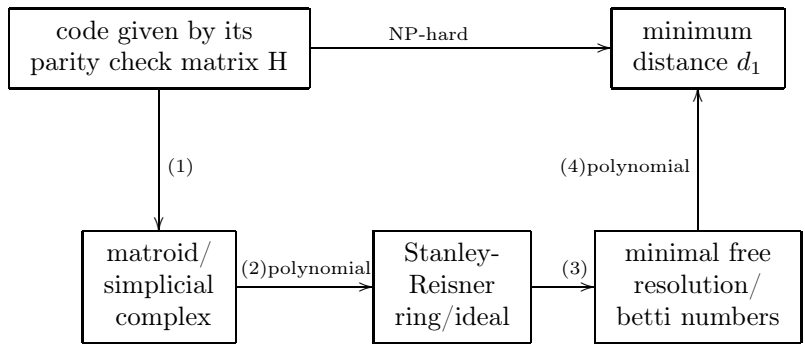
\[\xymatrix{*+[F]{\begin{array}{c}\textrm{code given by its }\\ \textrm{parity check matrix H}\end{array}} \ar[rr]^{\textrm{NP-hard}} \ar[dd]^{(1)} && *+[F]{\begin{array}{c}\textrm{minimum} \\\textrm{distance }d_1\end{array}} \\ \\*+[F]{\begin{array}{c}\textrm{matroid/}\\\textrm{simplicial}\\\textrm{ complex}\end{array}} \ar[r]^{\textrm{(2)polynomial}} &*+[F]{\begin{array}{c}\textrm{Stanley-}\\\textrm{Reisner}\\\textrm{ring/ideal}\end{array}} \ar[r]^{(3)} &*+[F]{\begin{array}{c}\textrm{minimal free}\\ \textrm{resolution/}\\\textrm{betti numbers}\end{array}} \ar[uu]_{\textrm{ }}^{\textrm{(4)polynomial}}}\]

As Steps $2$ and $4$ are clearly polynomial, Vardy' result should imply that Step $1$ or Step $3$ would be NP-hard (including the theoretical possibility that both of these steps are NP-hard). But we think that Step $1$ is NP-hard. This step almost amounts to finding all the codewords (almost because some codewords might have a support included in the support of another codeword - and a circuit of the matroid is a minimal support for the inclusion of the codewords). But that seems to us at least as hard as the minimum distance problem stated in \cite{V}. Since Step $1$ thus might be NP-hard, we are not able to draw any conclusion about the NP-hardness of Step $3$ from the result in \cite{V}.

It should also be remarked that in Proposition 2.9 of \cite{BS} one shows that to find the Krull dimension of $R=k[x_0,\cdots,x_n]/I$, for a monomial ideal $I$, is NP-complete. In our case, however, we just study those particular monomial ideals that are Stanley-Reisner ideals $I$ of simplicial complexes where the faces are precisely the independent sets of some matroid. For these particular rings $R$, finding the $\N$-graded Betti numbers will give all the non-zero ungraded Betti numbers, and the numbers of these will give the projective dimension, and by the Auslander-Buchsbaum theorem therefore the depth, and since matroid simplicial complexes are Cohen-Macaulay, also the Krull dimension. So if one could prove that finding the Krull dimension of those particular Stanley-Reisner rings we are dealing with, is NP-complete (or NP-hard), then finding the Betti numbers of those particular rings would also be NP-complete (and NP-hard), and finding the Betti numbers of all Stanley-Reisner rings would be NP-hard, as would be finding the Betti numbers of all monomial ideals. But unfortunately the proof that is applied in \cite{BS} to establish NP-completeness of deciding the Krull dimension of a quotient of a polynomial  ring by a general monomial ideal (reducing to the vertex cover problem described  in \cite{K}), does not work when restricting only to the subclass of the particular ideals that are Stanley-Reisner ideals of matroidal simplicial complexes.

%%%%%%%%%%%%%%%%%%%%%%%%%%%%%%%%%%%%%%%%%%%%%%%%%%%%%%%%%%%%%%%%%%%%%%%%%%%%%%%%%%%%%%%%%%%%%%%%%%%%%%%%%%%%
%
%  DEFINITIONS ET NOTATIONS
%
%%%%%%%%%%%%%%%%%%%%%%%%%%%%%%%%%%%%%%%%%%%%%%%%%%%%%%%%%%%%%%%%%%%%%%%%%%%%%%%%%%%%%%%%%%%%%%%%%%%%%%%%%%%

\section{Definitions and notation}

A simplicial complex $\Delta$ on the finite ground set $E$ is a subset of $2^E$ closed under taking subsets.  We refer to~\cite{1Miller} for a brief introduction of the theory of simplicial complexes, and we follow their notation. The elements of $\Delta$ are called faces, and maximal elements under inclusion, facets. The dimension of a face is equal to one less its cardinality. We denote $F_i(\Delta)$ the set of its faces of dimension $i$. The Alexander dual $\Delta^\star$ of $\Delta$ is the simplicial complex defined by \[\Delta^\star = \{ \overline{\sigma}: \sigma \not \in \Delta\}\] where $\overline{\sigma}=E-\sigma$. Given a simplicial complex $\Delta$ on the ground set $E$, define its Stanley-Reisner ideal and ring in the following way: let $\k$ a field and let $S=\k[\mathbf{x}]$ be the polynomial ring over $\k$ in $|E|$ indeterminates $\mathbf{x} = \{x_e:\ e \in E\}$. Then the Stanley-Reisner ideal $I_\Delta$ of $\Delta$ is   \[I_\Delta = <\mathbf{x}^\sigma|\ \sigma \not \in \Delta>\] and its Stanley-Reisner ring is $R_\Delta=S/I_\Delta$. This ring has a minimal free resolution as a $\N^{E}$-graded module \begin{equation} \label{SR} 0 \longleftarrow R_\Delta  \overset{\partial_0}{\longleftarrow} P_0  \overset{\partial_1}{\longleftarrow} P_1 \longleftarrow \ldots \overset{\partial_l}{\longleftarrow} P_l\longleftarrow 0 \end{equation} where each $P_i$ is of the form \[P_i = \bigoplus_{\alpha \in \N^{E}}S(-\alpha)^{\beta_{i,\alpha}}.\] 
The $\beta_{i,\alpha}$ are called the $\N^{E}$-graded Betti numbers of $\Delta$ over $\k$. We have $\beta_{i,\alpha}=0$ if $\alpha \in \N^E - \{0,1\}^E$. The Betti numbers are independent of the choice of the minimal free resolution. The $\N$-graded and global Betti numbers of $\Delta$ are respectively \[\beta_{i,d} = \sum_{|\alpha| = d}\beta_{i,\alpha}\] and \[\beta_i = \sum_d \beta_{i,d}.\]

For sake of clarity, we will picture the $\N$-graded Betti numbers in a Betti diagram, in which $\beta_{i,d}$ appears at the intersection of the $i$-th column and the $(d-i)$-th row. In these diagrams, we omit $\beta_{0,i}$, since $\beta_{0,0}=1$ and $\beta_{0,d}=0$ otherwise.

If we give an ordering $\omega$ on $E$, we can build the reduced chain complex of $\Delta$ with respect to $\omega$ in the following way: for any $ i \in \N$, let $V_i$ be a vector space over $\k$ whose basis elements $e_\sigma$ correspond to $\sigma \in F_i(\Delta)$. And if $\sigma \in F_i(\Delta)$, define \[\partial_{\omega,i}(e_\sigma) = \sum_{x \in \sigma}\epsilon_{\omega,\sigma}(x)e_{\sigma -\{x\}},\] where $\epsilon_{\omega,\sigma}(x) = (-1)^{r-1}$ if $x$ is the $r^{\text{th}}$ element in $\sigma$ with respect to the ordering $\omega$.  The chain complex is \[0 \longleftarrow V_{-1} \overset{\partial_{\omega,0}}\longleftarrow V_0 \overset{\partial_{\omega,1}}\longleftarrow V_1 \overset{\partial_{\omega,2}}\longleftarrow\ldots \overset{\partial_{\omega,|E|-1}}\longleftarrow V_{|E|-1} \longleftarrow 0.\] Let \[\tilde{h}_i(\Delta;\k) = \dim_\k(\ker(\partial_{\omega,i})/\text{im}(\partial_{\omega,i+1})).\] This is independent of the ordering $\omega$, so we omit it in the notation.

If $\sigma \subset E$, we denote by $\Delta_\sigma$  the simplicial complex whose faces are $\{\tau \cap \sigma:\ \tau \in \Delta\}$.  A result by Hochster~\cite{Hochster1} says that \[\beta_{i,\sigma} = \tilde{h}_{|\sigma|-i-1}(\Delta_\sigma;\k).\]

Let now $M$ be a matroid on the finite ground set $E$. A matroid is a simplicial complex satisfying the following extra property: if $\tau,\sigma$ are faces with $|\tau|<|\sigma|$, then there exists $x \in \sigma - \tau$ such that $\tau \cup\{x\}$ is a face. We refer to~\cite{1Oxley} for their theory. We denote by $\mathcal{I}_M$, $\mathcal{B}_M$, $\mathcal{C}_M$ and  $r_M$ the set of independent sets (faces), bases (facets), circuits (minimal non-faces) and rank function of $M$ respectively. The nullity function $n_M$ is defined by $n_M(\sigma)=|\sigma|-r_M(\sigma).$ By abuse of notation, $r(M)=r_M(E)$. The dual $\overline{M}$ of $M$ is defined by \[\mathcal{B}_{\overline{M}} = \{\overline{\sigma}:\ \sigma \in \mathcal{B}_M\}.\] Note that while $\overline{M}$ is a matroid, $M^\star$ generally isn't. Note also that if $\sigma \subset E$, then $M_\sigma$ is a matroid on the set $\sigma$, with rank function \[r_{M_\sigma}(\tau) = r_M(\tau).\]

\begin{definition} \label{matroidweights}
The generalized Hamming weights of $M$ are defined by \[d_i = \min\{|\sigma|:\ n_M(\sigma) = i\}\] for $1\leqslant i \leqslant |E| - r(M)$. 
\end{definition}
We will often omit the reference to $M$ when it is obvious from the context.
\begin{remark}\label{field} The Betti numbers considered in this paper are independent of the choice of the field $\k$. For matroids, Hochster's formula gives the Betti numbers. Since the restrictions of a matroid to subsets of the ground set are themselves matroids, and therefore shellable simplicial complexes, the reduced homology of these complexes, and thus the Betti numbers of the matroid, are independent of the choice of the field (\cite{Bjorner1}). For Alexander duals of matroids, this is~\cite[Corollary 5]{Eagon1}.
\end{remark}

\begin{remark} \label{central}
 If $C$ is a linear $[n,k]$-code over some finite field $\Fq$, with $(r \times n)$ parity check matrix $H$ (think of $r$ as redundancy of $C$ or rank of $H$), and 
$M=M_H$ is the matroid associated to $H$, then it is well known (see e.g. \cite{W}) that the higher weight heierarchy $d=d_1 < \ldots < d_k$ of $C$ as a linear code is identical to that of $M$ in the sense of Definition~\ref{matroidweights}, and this is the motivation of our defintion. Viewing the matroid $M$ as a special case of a so-called demi-matroid, as in \cite{BJMS}, the invariants $d_i$ are the same as those called $\overline{\sigma_i},$ for the trivial poset order there.
\end{remark} 

The goal of this paper is to see the relations between Betti numbers and generalized Hamming weights, for matroids in general, and then automatically, for those matroids associated to (parity check matrices of) linear codes.

Throughout this paper, we will use a running example, to emphasize the different points. Some other examples will also be provided. This running example is the following \begin{example} \label{runningexample}
Let $C$ be the binary non-degenerate $[6,3]$-code with parity check matrix \[H_1=\begin{pmatrix}1  & 0 & 0 & 1 & 0 & 1 \\0  & 1 & 0 & 1 & 1 & 0 \\0  & 0 & 1 & 1 & 1 & 0\end{pmatrix}.\] The matroid $M_1=M_{H_1}$ has $E=\{1,2,3,4,5,6\}$ and maximal independent sets (basis): \[ \B_1=\{\{1,2,3\}, \{1,2,4\}, \{ 1,2,5\}, \{ 1,3,4\}, \{ 1,3,5\}, \{ 2,3,4\}, \{2,3,6\}, \] \[ \{2,4,5\}, \{ 2,4,6\}, \{  2,5,6\}, \{ 3,4,5\}, \{ 3,4,6\}, \{ 3,5,6\}\}. \] Its circuits are \[\mathcal{C}_1= \{\{1,2,3,4\},\{1,4,5\},\{1,6\},\{2,3,4,6\},\{2,3,5\},\{4,5,6\}\}.\] Its Stanley-Reisner ring is $R_{M_1}=\k[x_1,x_2,x_3,x_4,x_5,x_6]/I_1$ where \[I=<x_1x_2x_3x_4, x_1x_4x_5,x_1x_6,x_2x_3x_4x_6, x_2x_3x_5,x_4x_5x_6>.\] The generalized Hamming weights are $d_1=2$, $d_2=4$ and $d_3=6$.
\end{example}

%%%%%%%%%%%%%%%%%%%%%%%%%%%%%%%%%%%%%%%%%%%%%%%%%%%%%%%%%%%%%%%%%%%%%%%%%%%%%%%%%%%%%%%%%%%%%%%%%%%%%%%%%%%%
%
%  ICI ON MONTRE QUE n(\sigma)=n SSI \sigma ADMET n CIRCUITS NON-REDUNDANTS ET PAS PLUS
%
%%%%%%%%%%%%%%%%%%%%%%%%%%%%%%%%%%%%%%%%%%%%%%%%%%%%%%%%%%%%%%%%%%%%%%%%%%%%%%%%%%%%%%%%%%%%%%%%%%%%%%%%%%%

\section{Relation between the nullity function and the non-redundancy of circuits}

We start by giving some definitions about the non-redundancy of circuits.

\begin{definition}
Given a matroid $M$ and $\Sigma \subset \mathcal{C}_M$. We say that the elements in $\Sigma$ are non-redundant if for every $\sigma \in \Sigma$, \[\bigcup_{\tau \in \Sigma-\{\sigma\}} \tau \subsetneq \bigcup_{\tau \in \Sigma} \tau.\]
\end{definition}

\begin{remark}\label{rquenr}It is obvious that the elements in $\Sigma$ are non-redundant if and only if for each $\sigma \in \Sigma$, there exists  $x_\sigma \in \sigma$ which isn't in any other $\tau \in \Sigma$.\end{remark}

\begin{definition}Let $M$ be a matroid and $\sigma$ be a subset of the ground set. The degree of non-redundancy of $\sigma$ is equal to the maximal number of non-redundant circuits contained in $\sigma$. It is denoted by $deg\ \sigma$.
\end{definition}

We will now see that the degree of non-redundancy of a subset is equal to its nullity.

\begin{lemma} Let $M$ be a matroid and let $\tau_1,\ldots,\tau_s$ be non-redundant circuits. Then \[n(\bigcup_{1\leqslant i\leqslant s}\tau_i) \geqslant s.\]
\end{lemma}

\begin{proof} This is obvious for $s=1$. There is actually equality in that case. Suppose that this is true for $s\geqslant 1$, and we prove that it is also true for $s+1$. Let $x_{s+1} \in \tau_{s+1}$ but in no other $\tau_i$. Of course, $\tau_1,\ldots,\tau_s$ are non-redundant, and by induction hypothesis, \[n(\bigcup_{1\leqslant i\leqslant s}\tau_i) \geqslant s.\] We know that for any given two subsets $A,B$ of the ground set, we have \[n(A\cup B) \geqslant n(A) + n(B) - n(A \cap B),\] since this is equivalent to the matroid axiom \[r(A\cup B) + r(A \cap B) \leqslant r(A) + r(B).\] If we apply it to $A= \bigcup_{1\leqslant i\leqslant s}\tau_i$ and $B =\tau_{s+1}$, noticing that $n(B) = n(\tau_{s+1}) = 1$, $n(A) \geqslant s$, we see that \[n(\bigcup_{1\leqslant i\leqslant s+1}\tau_i) \geqslant s +1 - n(A \cap B).\] But in this case, $A \cap B \subset \tau_{s+1} - \{x_{s+1}\}$ has to be independent, and therefore $n(A \cap B)=0$ and the lemma follows.
\end{proof}

\begin{corollary} Let $M$ be a matroid and $\sigma$ a subset of the ground set. Then \[n(\sigma) \geqslant deg(\sigma).\]
\end{corollary}
\begin{proof} Let $d=deg(\sigma)$, and $\tau_1,\ldots\tau_d$ be $d$ non-redundant circuits included in $\sigma$. Since $n$ is growing, we have \[n(\sigma) \geqslant n(\bigcup_{1\leqslant i\leqslant d}\tau_i) \geqslant d.\]
\end{proof}

\begin{lemma}\label{undeplus}Let $M$ be a matroid and $\tau_1,\ldots,\tau_m$ be $m$ non-redundant circuits with union $\tau=\bigcup_{1\leqslant i \leqslant m}\tau_i$. Let $\rho$ be another circuit such that $\rho \not\subset \tau$. Let $x \in \rho - \tau$. Then there exists a circuit $\tau_{m+1}$ such that $x \in \tau_{m+1}$ and such that $\tau_1,\ldots,\tau_{m+1}$ are non-redundant.
\end{lemma}

\begin{proof} 
Since the $\tau_i$ are non-redundant, for each $1\leqslant i \leqslant m$, by Remark~\ref{rquenr}, we can find $x_i \in \tau_i$ such that $x_i \not\in \tau_j$ if $j \neq i$. Consider the set of circuits that contain $x$. It is by hypothesis not empty. Consider an element $\tau_{m+1}$ in this set that contains fewest $x_i$. We claim that this number is $0$.  If not, then there exists $i \le m$ such that $x_i \in \tau_{m+1}$. Consider the two circuits $\tau_{m+1}$ and $\tau_i$. They have the element $x_i$ in common. Moreover, $x \in \tau_{m+1} - \tau_i$. By the strong elimination axiom for circuits of a matroid, we can find a circuit $\sigma$ such that \[x \in \sigma \subset \tau_{m+1} \cup \tau_i -\{x_i\}.\] It is easy to see that $\sigma$ has fewer $x_i$ than $\tau_{m+1}$, which is absurd. This means that $\tau_1,\ldots,\tau_{m+1}$ are non-redundant.
\end{proof}

\begin{corollary}\label{union_nonredundante} Let $M$ be a matroid, and $\tau_1,\ldots,\tau_m$ be a maximal set of non-redundant circuits. Then \[\bigcup_{1\leqslant i \leqslant m}\tau_i = \bigcup_{\tau \in \mathcal{C}_M} \tau.\]
\end{corollary}

\begin{lemma}Let $M$ be a matroid and $\sigma$ a subset of the ground set. Let $d=n(\sigma)$. Then there exist $d$ non-redundant circuits in $\sigma$. Thus $deg(\sigma) \geqslant n(\sigma)$.
\end{lemma}

\begin{proof}
If $n(\sigma)=0$, then $\sigma$ is independent and doesn't contain any circuit. If $n(\sigma)\ge 1$, then $\sigma$ is dependent and contains at least one circuit. This shows that the lemma holds for $d=0$ and for $d=1$. Suppose now that the lemma doesn't hold, and let $\sigma$ be minimal for the inclusion such that it doesn't hold. Then, by the previous remarks, $d= n(\sigma)\ge 2$ and we can find a circuit $\tau \subset \sigma$. As $\tau$ is not empty, choose $x \in \tau$ and consider $\sigma'=\sigma - \{x\}$. By minimality of $\sigma$, the lemma holds for $\sigma'$, and we can therefore find $n(\sigma')$ non-redundant circuits in $\sigma'$. Since \[d-1 \le n(\sigma') \le d,\] we have thus found at least $d-1$ non-redundant  circuits in $\sigma'$, and a fortiori in $\sigma$, say $\tau_1,\ldots,\tau_{d-1}$. Since \[x \in \tau - \bigcup_{1 \le i \le d-1}\tau_i,\] we can apply Lemma~\ref{undeplus} to get a circuit $\tau_d$ such that $\tau_1,\ldots,\tau_d$ are non-redundant, which is absurd. 
\end{proof}

\begin{proposition} Let $M$ be a matroid, and let $\sigma$ be a subset of the ground set. Then we have \[deg\ \sigma = n(\sigma).\]
\end{proposition}

\begin{example}For our running example (Example~\ref{runningexample}), let $\sigma=\{1,2,3,4,5,6\}$. Then $n(\sigma) = 3$. Even if we can write $\sigma$ as a union of $2$ non-redundant circuits \[\sigma = \{1,2,3,4\} \cup\{4,5,6\},\] it can also be written as a union of $3$ non-redundant circuits, for instance \[\sigma=\{1,2,3,4\} \cup \{1,4,5\} \cup \{1,6\}\] but not as a union of $4$ or more non-redundant circuits. 
\end{example}

%%%%%%%%%%%%%%%%%%%%%%%%%%%%%%%%%%%%%%%%%%%%%%%%%%%%%%%%%%%%%%%%%%%%%%%%%%%%%%%%%%%%%%%%%%%%%%%%%%%%%%%%%%%%
%
%  PARTIE PRINCIPALE: LES NOMBRES DE BETTI DONNENT LA HIERARCHIE
%
%%%%%%%%%%%%%%%%%%%%%%%%%%%%%%%%%%%%%%%%%%%%%%%%%%%%%%%%%%%%%%%%%%%%%%%%%%%%%%%%%%%%%%%%%%%%%%%%%%%%%%%%%%%

\section{Betti numbers and generalized Hamming weights}
\subsection{The $\N$-graded Betti numbers give the weight hierarchy}
Let $M$ be a matroid on the ground set $E$. For any integer $0\leqslant d \leqslant |E| - r(M)$, let $N_d = n^{-1}(d)$. Note that $N_0=\mathcal{I}$. We will now prove the following:

\begin{theorem} \label{prep}
Let $M$ be a matroid on the ground set $E$. Let $\sigma \subset E$. Then \[\beta_{i,\sigma} \neq 0 \Leftrightarrow \sigma \text{ is minimal in }N_{i}\textrm{ for inclusion}.\] Moreover,  \[\beta_{n(\sigma),\sigma} = (-1)^{r(\sigma)-1}\chi(M_\sigma).\]
\end{theorem}

\begin{proof} The matroid $M_\sigma$ has rank $r(\sigma)$, and thus by ~\cite[Th. 7.8.1]{Bjorner1}, we know that $M_\sigma$ might have reduced homology just in degree $r(\sigma)-1$.   We know that $\beta_{i,\sigma} = \tilde{h}_{|\sigma|-i-1}(M_\sigma,\k)$. So $\beta_{i,\sigma}=0$ except may be when $i=n(\sigma)$. In this case, we have by~\cite[Th 7.4.7 and 7.8.1]{Bjorner1}\[\beta_{i,\sigma} = \tilde{h}_{r(\sigma)-1}(M_\sigma,\k) = (-1)^{r(\sigma)-1}\chi(M_\sigma),\] for any field $\k$.\\ 
It remains to prove that this is non-zero if and only if $\sigma$ is minimal in $N_{i}$. It is well know that for a matroid $N$, $\chi(N) = 0$ if and only if $N$ has an isthmus, that is, if and only if $\overline{N}$ has a loop. This follows for example from  ~\cite[Exerc. 7.39]{Bjorner1}. But we have the equivalences: \begin{itemize} \item The matroid $\overline{N}$ has a loop, \item There exists an element which is in no base of $\overline{N}$, \item There exists an element which  is in all the  bases of $N$, \item There exists an element which is in no circuit of $N$, \item The underlying set of $N$ is not equal to the union of its circuits.  \end{itemize} And Corollary~\ref{union_nonredundante} just says that $\sigma$ is minimal in $N_i$ if and only if it is equal to the union of its circuits.  
\end{proof}

\begin{corollary} \label{betainfo}
\begin{itemize}
\item[a)]  Let $M$ be a matroid on the ground set $E$. Then \[\beta_{0,\sigma} = \left\{\begin{array}{ll} 1 & \text{ if } \sigma = \emptyset \\ 0 & \text{ otherwise }\end{array}\right.\] and \[\beta_{1,\sigma} =  \left\{\begin{array}{ll} 1 & \text{ if } \sigma \text{ is a circuit} \\ 0 & \text{ otherwise }\end{array}\right.\]
\item[b)] The resolution has length exactly $k=|E|-r(M)$, that is: $N_k \ne 0$, but $N_i =0$, for $i >k.$
\end{itemize}
\end{corollary}

\begin{proof}$ $
\begin{itemize}
\item[a)] is immediate from Theorem~\ref{prep}.
\item[b)] There exists a $\sigma$ such that $|\sigma| - rk(\sigma)=n-r$ (for example $\sigma =E$) but no $\sigma$ with $|\sigma| - r_M(\sigma) >n-r$. Hence $N_{n-r} \ne  \emptyset$, but $N_i=  \emptyset$ if $i > n-r$.
\end{itemize}
\end{proof}
\begin{corollary} A matroid $M$ is entirely given by its $\N^E$-graded Betti numbers in homology degree 1. Namely, we have \[ \mathcal{C}_M = \left\{\sigma \in E:\ \beta_{1,\sigma}=1\right\}.\]
\end{corollary}

\begin{remark} \label{CM}Part a) of corollary~\ref{betainfo} is just the standard interpretation of Betti numbers in terms of minimal generators of the ideal, which correspond to circuits of the matroid.

 Part b) of Corollary \ref{betainfo} is not new. It shows that the projective dimension of $R$ as an $S$-module is $\dim S - \dim R =n-r$. In \cite[Theorem 3.4]{Stan} it is shown that the Stanley-Reisner ring $R$ of a matroidal simplicial complex is level, in particular it is a Cohen-Macaulay graded algebra over $k$, and then the projective dimension is $n-r$. That $R$ is level also means that the rightmost term $P_{n-r}$ of its minimal resolution is pure, that is of the form $S(-b)^a$ for some non-negative integers $a,b$.
\end{remark}

We are now able to give and prove the following relation between $\N$-graded Betti numbers of a matroid $M$ and its Hamming weights:

\begin{theorem} \label{main}
Let $M$ be a matroid on the ground set. Then the generalized Hamming weights are given by \[d_i = \min\{d:\ \beta_{i,d} \neq 0\} \text{ for } 1\leqslant i \leqslant |E| - r(M).\]
\end{theorem}

\begin{proof} Let $\sigma$ minimal such that $n(\sigma)=i$. Then $\beta_{i,\sigma} \neq 0$ which implies $\beta_{i,d_i} = \beta_{i,|\sigma|} \neq 0$, and thus \[d_i \geqslant \min\{d:\ \beta_{i,d} \neq 0\}.\] Let now $d$ minimal such that $\beta_{i,d} \neq 0$. This means that there exists a subset $\sigma$ of $E$ of cardinality $d$ such that $\beta_{i,\sigma} \neq 0$. Then $\sigma$ is (minimal) in $N_d$, and thus \[d_i \leqslant \min\{d:\ \beta_{i,d} \neq 0\}.\] 
\end{proof}

\begin{corollary} 
Let $M$ be a matroid on the ground set $E$, of rank $r$. Then \[d_{|E|-r} = \left|\bigcup_{\tau \in \mathcal{C}} \tau\right| = |E|-|\{\text{loops of }\overline{M}\}|.\]
\end{corollary}

\begin{remark} \label{hs}
When $M=M_H$, the matroid of some parity check matrix for a linear code $C$, this number is just the cardinality of the support of $C$, since each loop of $\overline{M}=M_G$ for any generator matrix $G$ of $C$, corresponds to a coordinate position where all code-words are zero.

As another comment, not directly related to coding, we add that since $R$ (see Remark~\ref{CM}) is level, we have by \cite[Prop. 3.2,f.]{Stan} that $P_{n-r}=S(-d_{n-r})^{h_s}$, where $s$ is maximal such that $h_s \neq 0$. For a Cohen-Macaulay Stanley-Reisner ring the $h_i$ can be defined by $\sum_{i=0}^r f_{i-1}(t-1)^{r-i}= \sum_{i=0}^r h_it^{d-i},$ where $f_i$ is the number of independent sets of cardinality $i+1$ in the matroid $M$ (See \cite[Formula (1)]{Eagon1}). Here $s \le n-r$, and we see that \[h_{n-r}=\sum_{i=0}^r f_{i-1}(-1)^{r-i}.\]  The vector $(h_0,\ldots,h_s)$ is called the $h$-vector of the simplicial complex.
\end{remark}

\begin{example} \label{uniformexample}
It is well known (see e.g. \cite[Text following the proof of Lemma 5.1]{SW}) that the Betti diagram of the uniform matroid $U(r,n)$ corresponding to MDS-codes of length $n$ and dimension $k=n-r$ (since we are studying the rank function of the parity check matrix/matroid) is \[\begin{array}{c|ccccc} & 1&\cdots & s & \cdots& n-r\\\hline  r & \binom{r}{r}\binom{n}{r+1} & \cdots & \binom{r+s-1}{r}\binom{n}{r+s} & \cdots & \binom{n-1}{r}\binom{n}{n}
\end{array}\] Hence the weight hierarchy is $\{n-k+1,\ldots,n-1,n\}.$ We see from the rightmost part of the resolution that $h_{n-r}=\binom{n-1}{r}$, while \[f_{r-1}-f_{r-2}+\ldots+(-1)^{r}f_{-1}=\binom{n}{r}-\binom{n}{r-1}+\ldots+(-1)^{r}\binom{n}{0},\] which is also $\binom{n-1}{r}$.
\end{example}

\begin{example} \label{63example} For our running example (Example~\ref{runningexample}), using~\cite{BCP}, we get the Betti diagram \[\begin{array}{c|ccc} & 1&2&3 \\\hline 1 &1\\ 2&3&2\\ 3&2&7&4\end{array}\] Hence the $d_i$ are $2,4,6.$
\end{example}

\subsection{Negative results for the converse and for global Betti numbers}
We will now give examples showing that the converse of Theorem~\ref{main} doesn't hold, in the sense that the generalized Hamming weights of a matroid do not determine its $\N$-graded Betti-numbers. Moreover it doesn't hold even if we replace $\N$-graded Betti numbers by global Betti numbers. While the theory is valid for matroids in general, and thus for (parity check) matroids defined by linear codes in particular, the negative examples that we give, all come from codes.

We start by giving an example showing that the weight hierarchy of a code doesn't in general determine the global Betti numbers (and therefore not the $\N$-graded Betti numbers either).

\begin{example} \label{counter}
Consider the binary non-degenerate $[4,2]$-codes with parity check matrices \[H_2=\begin{pmatrix}1&0&1&1\\0&1&1&1\end{pmatrix}\] and \[H_3=\begin{pmatrix}1&0&0&1\\0&1&1&0\end{pmatrix}.\] Their associated matroids are on $E=\{1,2,3,4\}$ with basis sets  \[\mathcal{B}_{2} = \{\{1,2\},\{1,3\},\{1,4\},\{2,3\},\{2,4\}\}\] and \[\mathcal{B}_{3} = \{\{1,2\},\{1,3\},\{2,4\},\{3,4\}\}\] respectively. Their weight hierarchy is $(2,4)$, while their global Betti numbers are $(1,3,2)$ and $(1,2,1)$ respectively.
\end{example}

The next example shows that two codes with the same global Betti numbers might have different weight hierarchies.

\begin{example}\label{rque44}
Consider the binary non-degenerate $[4,2]$-codes with parity check matrices \[H_4=\begin{pmatrix}1&0&0&0\\0&1&1&1\end{pmatrix}\] and \[H_5=\begin{pmatrix}1&1&0&1\\0&1&1&1\end{pmatrix}.\]Their associated matroids are on $E=\{1,2,3,4\}$ with basis sets \[\mathcal{B}_4 = \{\{1,2\},\{1,3\},\{1,4\}\}\] and \[\mathcal{B}_5=\{\{1,2\},\{1,3\},\{1,4\},\{2,3\},\{3,4\}\}\] respectively. Their global Betti numbers are $(1,3,2)$ while their weight hierarchies are $(2,3)$ and $(2,4)$ respectively.
\end{example}

\subsection {Alexander duality}

From Wei duality generalized to matroids (\cite{BJMS}) it is clear that the higher weight hierarchy of a matroid $M$ determines and is determined by that of its matroid dual $\overline{M}$. But since $M$ is also  a simplicial complex, through its set $\I_M$ of independent sets, there is also another duality, Alexander duality, which comes into play. Since~\cite{Eagon1}, it is well known that for the underlying simplicial complex of $M$, the $\N$-graded resolution of the Stanley-Reisner ring of the Alexander dual complex, $M^\star$, has a particularly nice form. Indeed, as $M$ is Cohen-Macaulay, the resolution is linear and of the form \[0 \longleftarrow R_{M^\star} {\longleftarrow} S  {\longleftarrow} P_1 \longleftarrow \ldots {\longleftarrow} P_l\longleftarrow 0\] where each $P_i$ is of the form \[P_i = S(-r(\overline{M})-1-i)^{\beta_{i}},\] for some $l$.

We thus study the minimal resolutions of the Stanley-Reisner rings of the Alexander duals $M^\star$ or ${\overline{M}}^\star$ and investigate whether the Betti numbers $\beta_i$ of such resolutions determine the higher weight hierarchy of $M$ (and $\overline{M}$).

Unfortunately it is clear however that even for two matroids of the same cardinality $n$ and rank $k$ the Betti-numbers of $M^\star$ and/or ${\overline{M}}^\star$ do not in general determine the higher weight hierarchy. Indeed, Formula $1$ and the second part of Theorem 4 of~\cite{Eagon1} show that determining the Betti numbers of the Alexander dual of the matroid is equivalent to finding the $f$-vector of the matroid. Finding the $f_j$, although an important characterization of a matroid, is not enough to find the higher weights of a code giving rise to the matroid (if such a code exists).

An even more important characterization of a matroid is its Whitney polynomial \[W(x,y) = \sum_{X \subset E}x^{r(E)-r(X)}y^{|X|-r(X)}.\]

The information about the $f_i$ can then be read off from the coefficients of the pure $x$-part $W(x,0)$ of the Whitney polynomial. How the $d_i$ can be read off is described in~\cite[p. 131]{D}.

\begin{example}For our running example (Example~\ref{runningexample}), its Alexander dual has the facets \[\{\{1,2,3\},\{1,4,6\},\{1,5\},\{2,3,4,5\},\{2,3,6\},\{5,6\}\},\] and its Betti diagram is linear
\[\begin{array}{c|cccc}&1&2&3&4\\\hline 2&13&25&17&4\end{array}.\] 
The Whitney polynomial of the matroid $M_1$ is \[W(x,y)=x^3+x^2y+6x^2+xy^2+7xy+14x+y^3+6y^2+14y+13.\] Then \[W(x,0)= x^3+6x^2+14x+13\] and from there, we can read that \[(d_1,d_2,d_3) = (2,4,6) \textrm{ and } (f_0,f_1,f_2)=(6,14,13).\]\end{example}
 Even if the procedures for reading off the $d_j$ from $W(x,y)$ are very different, it could of course a priori happen that the $\beta_i$ determined the $d_j$. The following example shows, however, that  Alexander duals of matroids $M$ may have the same $\beta_i$, but different $d_j$

\begin{example}\label{exa51}
Consider the non-degenerate $[6,3]$-code over $\F{5}$ with parity check matrices \[H_8=\begin{pmatrix}0  & 1 & 0 & 0 & 1 & 1 \\1  & 0 & 1 & 0 & 1 & 1 \\0  & 0 & 1 & 1 & 0 & 1\end{pmatrix}\]and \[H_9=\begin{pmatrix}0  & 1 & 0 & 1 & 1 & 1 \\0  & 0 & 1 & 1 & 2 & 3 \\1  & 0 & 0 & 0 & 0 & 1\end{pmatrix}\]Their associated matroids $M_8=M_{H_8}$ and $M_9=M_{H_9}$ are on $E=\{1,2,3,4,5,6\}$ with the basis sets  \[\B_8=\{\{1,2,3\}, \{1,2,4\},\{1,2,6\},\{1,3,5\},\{1,3,6\},\{1,4,5\},\{1,4,6\},\{1,5,6\},\]
\[\{2,3,4\},\{2,3,5\},\{2,4,5\},\{2,4,6\},\{2,5,6\},\{3,4,5\},\{3,4,6\},\{3,5,6\} \}\]  and
\[\B_9=\{\{1,2,3\}, \{1,2,4\},\{1,2,5\},\{1,2,6\},\{1,3,4\},\{1,3,5\},\{1,3,6\},\{1,4,5\},\]
\[\{1,4,6\},\{1,5,6\},\{2,3,6\},\{2,4,6\},\{2,5,6\},\{3,4,6\},\{3,5,6\},\{4,5,6\} \}\] 
respectively. They both give rise to the following Betti diagram for $R_{(M_i)^\star}$, $i \in \{8,9\}$ \[\begin{array}{c|cccc} & 1&2&3&4 \\\hline 2& 16&33&24&6\end{array}\] while their respective weight hierarchies are $(3,5,6)$ and $(2,5,6)$.
\end{example}

Also it is clear from the following very simple examples that the $d_j$ do not in general determine the $\beta_i$ (of the Alexander dual of $M$):

\begin{example}\label{exa52}
Consider the codes from Example~\ref{counter}. We have already seen that they have the same weight hierarchies. But 
the Stanley-Reisner rings of the Alexander dual simplicial complexes $M_2^\star$ and $M_3^\star$ give rise to the Betti diagrams \[\begin{array}{c|ccc} & 1&2&3 \\\hline 1&5&6&2\end{array}\] and \[\begin{array}{c|ccc} & 1&2&3 \\\hline 1&4&4&1\end{array}\] respectively.
\end{example}

\section{Consequences of the previous results}

In this section we derive and study some consequences of Theorem~\ref{main}. Recall that a linear code $C$ of length $n$ and dimension $k$ is called $h$-MDS if its higher support weight $d_h$ satisfies $d_h=n-k+h$ for some $h \in \{1,\ldots,k\}.$ If $C$ is $h$-MDS, then $C$ is $i$-MDS for all $i \in \{h,\ldots,k\}.$ We recall that  $C$ is called MDS if $C$ is $1$-MDS, that is $d=n-k+1$. For a length $n$ and dimension $k$ code, the matroid $M=M(C^\perp)$ has rank $r=n-k$ and is equipped with invariants $d_1,\ldots, d_k$ which have the same values as the support weights of $C$. For the resolution (\ref{SR}) of the Stanley-Reisner ring $R$ of (the simplicial complex of independent sets of) $M$ we then have:

\begin{corollary} \label{rMDS}We have the following:
\begin{itemize}
\item[a)] $C$ is $h$-MDS if and only if the right part \[P_h {\longleftarrow} P_{h+1} {\longleftarrow} \ldots {\longleftarrow} P_k \]of the resolution is linear, and $M$ has no isthmus.
\item[b)] $C$ is MDS if and only if the entire resolution is linear (that is linear from $P_1$ and rightover), and $M$ has no isthmus ($C$ is non-degenerate). 
\item [c)] If $C$ is non-degenerate, then it is MDS if and only if  the Alexander dual of $M$ is also (the set of independent sets of) a matroid.
\end{itemize}
\end{corollary}

\begin{proof}
It is clear from Corollary~\ref{betainfo} and the fact that $N_i=\emptyset$ for all $i >k=|E|-rk(E)$ that $\beta_{i,j}=0$ for all $j$ for $i >k$. From Theorem~\ref{main} it is clear that $d_h=n-k+h$ if and only if $\beta_{i,j}=0$ for $j < n-k+i$, and $\beta_{i,n-k+i} \ne 0,$ for each $i \in \{h,h+1,\ldots,k\}$ (in particular  for $i=h$, and thus implying the statements for the remaing $i$, by standard facts from coding theory).  Since there is no facet with more than $n$ elements it is clear that $\beta_{k,j}=0$ for $j > n.$  Let \[\nu_i=\min \{j : \beta_{i,j} \ne 0 \}=d_i \textrm{ and } \mu_i=\max \{j : \beta_{i,j} \ne 0 \}.\] We have just shown that $\nu_k=\mu_k=n$ if and only if  $d_k=n$. Since by \cite[Theorem 3.4]{Stan} $R$ is a level ring, in particular Cohen-Macaulay, it is true that the $\mu_i$ form a strictly increasing sequence (See \cite[Prop. 1.1]{BH}). It is therefore  also clear that if $M$ has no isthmus, that is $d_k=n$, then $\mu_i \le \mu_{k}-[k-i]=n-k+1.$ Hence $\beta_{i,j}=0$ for $j > n-k+i.$   This gives a). The statement of b) is just the special case $h=1$  of a).\\
 It is clear that if $C$ is an $[n,k]$ MDS-code, then $M$ is the uniform matroid $U(r,n)$, and then its Alexander dual is the matroid $U(k-1,n)$. If on the other hand $C$ is not MDS, then the Stanley-Reisner ring does not have a linear resolution, and then it's not the Alexander dual of any matroid (by \cite[Proposition 7]{Eagon1}).
\end{proof}

\begin{example} \label{isthmusexample}$ $
\begin{itemize}
\item Let $C$ be the linear binary $[6,3]$-code with parity check matrix \[H_6=\begin{pmatrix}1  & 0 & 0 & 1 & 1 & 0 \\0  & 1 & 0 & 0 & 1 & 1 \\0  & 0 & 1 & 1 & 0 & 1\end{pmatrix}\] For the matroid $M_6=M_{H_6}$ we have $E=\{1,2,3,4,5,6\}$ and maximal independent sets \[\B_6=\{\{1,2,3\}, \{1,2,4\}, \{ 1,2,6\}, \{ 1,3,5\}, \{ 1,3,6\}, \] \[ \{1,4,5\}, \{ 1,4,6\}, \{1,5,6\}, \{2,3,4\} ,\{2,3,5\}, \{2,4,5\},\] \[ \{2,4,6\}, \{2,5,6\}, \{3,4,5\}, \{3,4,6\}, \{3,5,6\}\}.\] 
It gives rise to the Betti diagram \[\begin{array}{c|ccc}&1&2&3\\\hline 2&4 \\ 3&3&12&6\end{array}\] Here $\{d_1,d_2,d_3\}=\{3,5,6\}$, so $C$ is $2$-MDS, and we see that the part of the resolution consisting of the two rightmost terms (corresponding to $d_2$ and $d_3$) is linear as described.

\item Let $C$ be the linear $[6,3]$-code over $\F{5}$ with generator matrix \[G_7=\begin{pmatrix}1  & 0 & 1 & 1 & 1 & 1 \\0  & 0 & 1 & 2 & 3 & 4 \\0  & 0 & 1 & 4 & 4 & 1\end{pmatrix}.\]Here the maximal independent sets of $M_{G_7}$ are precicely the $10$ subsets of $\{1,2, 3,4,5,6\}$ of cardinality $3$, not containing $2$. Let $H_7$ be any parity check matrix. For the matroid $M_7=M_{H_7}$ we have $E=\{1,2,3,4,5,6\}$ and the maximal independent sets are precisely the $10$ subsets of $E$ of cardinality $3$ containing $2$. Using \cite{BCP} this gives the 
Betti diagram \[\begin{array}{c|ccc}&1&2&3\\\hline 2&10&15&6\end{array}\] Hence the $d_i$ are $3,4,5.$ We see that $d_3=5,$ and not $6,$ since $C$ is degenerate, and  $M_7$ has the isthmus $2$ (loop of $\overline{M_7}$). The resolution is linear, but $M_7$ does not correspond to an MDS-code. The code obtained by truncating the second position is MDS of word length $5$.
\end{itemize}
\end{example}

\begin{example} \label{canonical}
Let $X$ be an algebraic curve of genus $g$ defined over $\Fq$, and embed $X$ into $\Proj^{g-1}$ by use of the complete linear system $L(K)$. Let $\{P_1,\ldots,P_n\}$ be a set of $\Fq$-rational points (of degree $1$) on $X$. We form the matrix $H$ where the $i$-th column consists of  the coordinates of (the image of) $P_i$ for each $i=1,\ldots,n.$ Each column is then determined up to a non-zero multiplicative constant; fix a choice for each $i$. Now we let $H$ be the parity check matrix of a linear code $C$, and let $M$ be the matroid associated to $C$. Different choices of multiplicative constants give equivalent linear codes and therefore the same code parameters, and even the same matroid $M$. If the chosen points fail to span all of  $\Proj^{g-1}$, we replace $H$ with a suitable matrix with fewer rows (that are linear combinations of those in $H$).  Set $D=P_1+\ldots+P_n$. (The code is also code equivalent to the algebraic-geometric code $C(D,D)$ in standard terminology, provided one is able to define such a code properly. As one understands, this is not a strongly algebraic-geometric code $C(D,G)$, since for such a code one demands  $2g-2 < deg(G) < deg(D)$.) Let the ground set $E$ be the set of subdivisors of $D$, corresponding to all subsets of $\{1,2,\ldots,n\}$, representing sets of columns of $H$. Let $A$ be a subdivisor of $D$, and let $r(A)$ be the value at $A$ of the rank function associated to the matroid $M$. It is a consequence of the geometric version of the Riemann-Roch theorem that  \[r(A)=l(K)-l(K-A)=g-h^1(A).\] Moreover, for the nullity $n(A)$, the Riemann-Roch theorem gives   \[n(A)=l(A)-1.\] These rank and nullity functions are described in detail in \cite[Section 5]{D}, which provides the inspiration for this example. 

We define the $t_D$-gonality of $X$ as the minimal degree of a subdivisor $A$ of $D$ such that $l(A)=t+1.$ Hence the $t_D$-gonality of $X$ is the minimal cardinality of a subset $A \subset E$ such that $n(A)=t$, in other words $d_t$. By Theorem~\ref{main}, $t_D=\min\{j : \beta_{t,j} \neq 0\}.$

We also define the $D$ Clifford index $Cl_D(A)$ of a subdivisor $A$ of $E$ as $deg(A)-2(l(A)-1)$. Regarding $A$ as a subset of $\{1,2,\ldots,n\}$ we obtain that this number is $|A|-2n(A)$. The $D$ Clifford index $Cl_D(X)$ of $X$ is $\min\{Cl_D(A)\}$, where $A$ is taken only over those $A$ with $h^0(A) \ge 2$ and $h^1(A) \ge 2$. We have $h^0(A)=n(A)+1 \ge 2$ if and only if $n(A) \ge 1$, and $h^1(A)=l(K-A)=l(A)-|A|-1+g= n(A)-|A|+g \ge 2$ if and only if $|A| \le n(A)+g-2$. This, in combination with Theorem~\ref{main}, gives \[Cl_D(X) = \min \{d_i-2i : i \ge 1; j \le g-2+i \}. \] Hence these kinds of Clifford indices can be read off from these kinds of  Betti numbers. The most interesting case is perhaps when one lets $D$ be the sum of all $\Fq$-rational points of $X$ (and the rank of the matroid is typically $g$ then). Then the $t_D$ and $Cl_D$ are close to being the usual $t$-gonality and Clifford index of $X$ restricted to $\Fq$. But these usual definitions also include divisors with repeated points.

Another example is $D=K$ as in \cite[Section 5]{D}. Then $M$ is a self dual matroid of rank $g-1$ (so $H$ will have to be processed a little to be a parity check matrix of the code). In \cite{D} one shows that $Cl_K(A) \ge 0$ for all $A$, using only properties of matroids.   

Since the $t_D$-gonalities thus have natural generalizations to all finite matroids and linear codes in form of the $d_t$, one might ask if the Clifford index also has such a generalization. Since $r=g$ for the particular matroid above (assuming the images of the points of $\textrm{Supp}(D)$ span $\Proj^{g-1}$), one might define $Cl(M)$ of any matroid $M$ as \[Cl(M) = \min \{d_i-2i : i \ge 1; j \le r-2+i \}. \] This is, however, only defined if the set we are taking the minimum over, is non-empty, and this happens if and only if $d \le r-1$. The Singleton bound only gives  $d \le r+1$. Hence this is not defined for MDS-codes (uniform matroids, $d=r+1$), and almost-MDS-codes (almost-uniform matroids, $d=r$). It is unclear to us whether such an Clifford index says something useful and/or interesting about linear codes or matroids, and in that case, if its definition can be relaxed to apply to MDS-codes and almost-MDS codes also. In general such a $Cl(M)$ can be negative. (Take a linear code with only zeroes in $2$ positions, but MDS after these two positions have been truncated. Then $d=r-1$, and for any parity check matrix $H$ we see that $C(M_H)$ is computed by $d_{n-r}-2(n-r)=2r-n-2$, which is negative for some $r,n$.)
\end{example}

%%%%%%%%%%%%%%%%%%%%%%%%%%%%%%%%%%%%%%%%%%%%%%%%%%%%%%%%%%%%%%%%%%%%%%%%%%%%%%%%%%%%%%%%%%%%%%%%%%%%%%%%%%%%
%
%  ON RESUME LE TOUT AVEC REFERENCE AUX THEOREMES ET CONTRE-EXEMPLES
%
%%%%%%%%%%%%%%%%%%%%%%%%%%%%%%%%%%%%%%%%%%%%%%%%%%%%%%%%%%%%%%%%%%%%%%%%%%%%%%%%%%%%%%%%%%%%%%%%%%%%%%%%%%%

\section{Conclusion}
 Summing up we can roughly say that there are $12$ sets of Betti-numbers that we have considered in this paper:
The three sets of  $\N^{E}$-graded, $\N$-graded, and  global Betti numbers, for each of the four simpicial complexes
$M, \overline M, M^\star$, and $({\overline M})^\star$.
Two of these sets, the  $\N^{E}$-graded ones for $M, \overline{M}$ determine $M$, and therefore the weight hierarchy  in a trivial way, since $\beta_{1,\sigma} \neq 0$ if and only if $\sigma$ is a circuit of the matroid in question.

Likewise two other sets, the  $\N^{E}$-graded ones for $M^\star, ({\overline M})^\star$ determine $M$, and therefore the weight hierarchy in a trivial way, since $\beta_{1,\sigma} \neq 0$ if and only if $\sigma$ is a basis of the matroid dual of the matroid in question.

Two other sets,   $\N$-graded ones for $M, \overline M$ determine the $d_i$ in (what we dare to consider) a non-trivial way, and this is the main result of our paper (Theorem~\ref{main}).

The two sets of global Betti numbers for  $M, \overline M$ do not in general determine the $d_i$, since we have presented examples of pairs of codes with different sets of $d_i$, but the same Betti numbers (Example~\ref{rque44}).

The two sets of  $\N$-graded Betti-numbers for the Alexander duals $M^\star, ({\overline M})^\star $ do not in general determine the $d_i$, since we have presented examples of pairs of codes with different sets of $d_i$, but the same sets of Betti numbers (Example~\ref{exa51}).

The two sets of  global Betti-numbers for the Alexander duals $M^\star, ({\overline M})^\star$ are the same as the sets of $\N$-graded ones (since the resolutions are linear), so the same conclusion applies to them.

Finally, Example~\ref{exa52} and Example~\ref{counter} show that the weight hierarchy does not always decide the Betti numbers.

%%%%%%%%%%%%%%%%%%%%%%%%%%%%%%%%%%%%%%%%%%%%%%%%%%%%%%%%%%%%%%%%%%%%%%%%%%%%%%%%%%%%%%%%%%%%%%%%%%%%%%%%%%%%
%
%  REMERCIEMENTS ET BIBLIOGRAPHIE
%
%%%%%%%%%%%%%%%%%%%%%%%%%%%%%%%%%%%%%%%%%%%%%%%%%%%%%%%%%%%%%%%%%%%%%%%%%%%%%%%%%%%%%%%%%%%%%%%%%%%%%%%%%%%

\subsubsection*{acknowledgements}
Trygve Johnsen is grateful to Institut Mittag-Leffler where he stayed during part of this work.\\
The authors are grateful to the referees and editor for their suggestions for improving the original paper.

\end{document}